\documentclass[12pt]{amsart}
\usepackage{amsmath}
\usepackage{amssymb}
\usepackage{amscd}
\usepackage[mathscr]{eucal}
\usepackage{latexsym}

\newtheorem{thm}{Theorem}[section]
\newtheorem{lem}[thm]{Lemma}

\newtheorem{prop}[thm]{Proposition}

\theoremstyle{definition}

\newtheorem{rem}[thm]{Remark}

\newtheorem{defn}[thm]{Definition}

\newtheorem{ex}[thm]{Example}

\theoremstyle{remark}

\numberwithin{equation}{section}

\def\tr{\text{\rm tr}}

\def\dR{\text{\rm dR}}
\def\HT{\text{\rm HT}}

\def\BK{\text{\rm BK}}
\def\Sen{\text{\rm Sen}}
\def\cris{\text{\rm cris}}

\def\log{\text{\rm log}}

\def\gr{\text{\rm gr}}

\def\st{\text{\rm st}}

\def\Frac{\text{\rm Frac}}
\def\Fil{\text{\rm Fil}}
\def\dim{\text{\rm dim}}

\def\Gal{\text{\rm Gal}}
\def\sup{\text{\rm sup}}

\def\Ker{\text{\rm Ker}\,}

\def\lim{\text{\rm lim}}

\begin{document}

\title[Generalization of the theory of Sen]
{Generalization of the theory of Sen in the semi-stable representation case}

\author[Kazuma Morita]{Kazuma Morita}
\address{Department of Mathematics, Hokkaido University, Sapporo 060-0810, Japan}
\email{morita@math.sci.hokudai.ac.jp}

\date{\today}

\maketitle
\begin{abstract}For a semi-stable representation $V$, we will construct a subspace $D_{\pi\text{-}\Sen}(V)$ of $\mathbb{C}_{p}\otimes_{\mathbb{Q}_{p}}V$ endowed with a linear derivation $\nabla^{(\pi)}$. The action of $\nabla^{(\pi)}$ on $D_{\pi\text{-}\Sen}(V)$ is closely related to the action of the monodromy operator $N$ on $D_{\st}(V)$. Furthermore, in the geometric case, the action of $\nabla^{(\pi)}$ on $D_{\pi\text{-}\Sen}(V)$ describes an analogy of the infinitesimal variations of Hodge structures and satisfies formulae similar to the Griffiths transversality and the local monodromy theorem.
\end{abstract}
\section{Introduction}
Let $K$ be a complete discrete valuation field of characteristic $0$ with 
perfect residue field $k$ of characteristic $p > 0$. Choose an
algebraic closure $\overline{K}$ of $K$ and consider its $p$-adic completion $\mathbb{C}_{p}$. By a $p$-adic representation
of $G_K=\Gal(\overline{K}/K)$, we mean a finite dimensional vector space $V$ over $\mathbb{Q}_p$ endowed with a continuous
action of $G_K$.
Put $K_{\infty}=\cup_{0\leq m}K(\zeta_{p^{m}})$ where $\zeta_{p^{m}}$ denote a primitive $p^{m}$-th root of unity in $\overline{K}$ satisfying $(\zeta_{p^{m+1}})^{p}=\zeta_{p^{m}}$. Let $H_{K}$ denote the kernel of the cyclotomic character $\chi:G_{K}\rightarrow \mathbb{Z}_{p}^{*}$ and define $\Gamma_{K}$ to be $G_{K}/H_{K}\simeq \Gal(K_{\infty}/K)$. 
Then, for a $p$-adic representation $V$ of $G_{K}$, Sen constructs a $K_{\infty}$-vector space $D_{\Sen}(V)$ of dimension $\dim_{\mathbb{Q}_{p}}V$  in $(\mathbb{C}_{p}\otimes_{\mathbb{Q}_{p}}V)^{H_{K}}$ equipped with the $K_{\infty}$-linear derivation $\nabla^{(0)}$ which is the $p$-adic Lie algebra of $\Gamma_{K}$.
In the case when $V$ is a Hodge-Tate representation of $G_{K}$, the set of eigenvalues of $\nabla^{(0)}$ on $D_{\Sen}(V)$ is exactly the same as the set of Hodge-Tate weights of $V$.

Now, we shall state the aim of this article. First, let us fix some notations. Fix a prime $\pi$ of $\mathscr{O}_{K}$ (the ring of integers of $K$) and for each $1\leq m$, fix a $p^{m}$-th root $\pi^{1/p^{m}}$ of $\pi$ in $\overline{K}$ satisfying $(\pi^{1/p^{m+1}})^{p}=\pi^{1/p^{m}}$. Put 
$K^{\BK}=\cup_{0\leq m}K(\pi^{1/p^{m}}) \ \text{and} \  K_{\infty}^{\BK}=\cup_{0\leq m}K_{\infty}(\pi^{1/p^{m}})$. Here, the letter BK stands for the Breuil-Kisin extension.
Let $H$ denote the Galois group $\Gal(\overline{K}/K_{\infty}^{\BK})$ and define $\Gamma_{\BK}$ to be   $\Gal(K_{\infty}^{\BK}/K_{\infty})$.
Then, we have an isomorphism of profinite groups $G_{K}/H\simeq \Gamma_{K}\ltimes \Gamma_{\BK}.$
In this article, for a semi-stable representation $V$ of $G_{K}$, we shall construct a $K_{\infty}^{\BK}$-vector space $D_{\pi\text{-}\Sen}(V)$ of dimension $\dim_{\mathbb{Q}_{p}}V$ in $(\mathbb{C}_{p}\otimes_{\mathbb{Q}_{p}}V)^{H}$ equipped with the $K_{\infty}^{\BK}$-linear derivations $\nabla^{(0)}$ and $\nabla^{(\pi)}$. Here, $\nabla^{(\pi)}$ denotes the $p$-adic Lie algebra of $\Gamma_{\BK}$.
Then, the action of $\nabla^{(0)}$ on $D_{\pi\text{-}\Sen}(V)$ tells us about Hodge-Tate weights as in the case of $D_{\Sen}(V)$ and the action of $\nabla^{(\pi)}$ on $D_{\pi\text{-}\Sen}(V)$ is closely related to the action of the monodromy operator $N$ on $D_{\st}(V)$. Furthermore, in the case $V=H_{\text{\'et}}^{m}(X\otimes_{K}\overline{K},\mathbb{Q}_{p})$ where $X$ denotes a proper smooth scheme over $K$, the action of $\nabla^{(\pi)}$ on $D_{\pi\text{-}\Sen}(V)$ describes an analogy of the infinitesimal variations of Hodge structures and satisfies formulae similar to the Griffiths transversality and the local monodromy theorem. 

{\bf Acknowledgments} 
The author would like to thank his advisor Professor
Kazuya Kato for continuous advice, encouragements and patience. 
He is also grateful to Professor Masanori Asakura, Olivier Brinon and Takeshi Tsuji for useful discussions.
A part of this work was done while he was staying at Universit\'e Paris-Sud 11 and
he thanks this institute for the hospitality.
His staying at Universit\'e Paris-Sud 11 was partially 
supported by JSPS Core-to-Core Program
``New Developments of Arithmetic Geometry, Motives, Galois Theory, and Their Practical Applications''
and he thanks Professor Makoto Matsumoto for encouraging this visiting.   
This research was partially supported by JSPS Research Fellowships for Young Scientists.
 
\section{Preliminaries on basic facts} 
\subsection{$p$-adic periods rings and $p$-adic representations}
(See [F1] for details.) 
Let $K$ be a complete discrete valuation field of characteristic $0$ with perfect residue field $k$ of characteristic $p>0$.
Put $K_{0}=\Frac(W(k))$ where $W(k)$ denotes the ring of Witt vectors with coefficients in $k$. 
Choose an algebraic closure $\overline{K}$ of $K$ and consider its $p$-adic completion $\mathbb{C}_p$.
Put 
$$\widetilde{\mathbb{E}}=\underleftarrow{\lim}_{x\mapsto x^{p}} \mathbb{C}_p=\verb+{+(x^{(0)},x^{(1)},...)\ |\ (x^{(i+1)})^p=x^{(i)},  x^{(i)}\in\mathbb{C}_{p} \verb+}+.$$
For two elements $x=(x^{(i)})$ and $y=(y^{(i)})$ of $\widetilde{\mathbb{E}}$, define their sum and product by
$(x+y)^{(i)}=\lim_{j\rightarrow\infty}(x^{(i+j)}+y^{(i+j)})^{p^j}$ and $(xy)^{(i)}=x^{(i)}y^{(i)}.$
Let $\epsilon=(\epsilon^{(n)})$ denote an element of $\widetilde{\mathbb{E}}$ such that $\epsilon^{(0)}=1$ and $\epsilon^{(1)}\not =1$.
Then, $\widetilde{\mathbb{E}}$ is a perfect field of characteristic $p>0$  and is the completion of an algebraic closure of $k((\epsilon-1))$ for the valuation defined by $v_{\mathbb{E}}(x)=v_{p}(x^{(0)})$ where $v_{p}$ denotes the $p$-adic valuation of $\mathbb{C}_{p}$ normalized by $v_{p}(p)=1$. The field $\widetilde{\mathbb{E}}$ is equipped with an action of a Frobenius $\sigma$ and a continuous action
of the Galois group $G_K=\Gal(\overline{K}/K)$ with respect to the topology defined by the valuation $v_{\mathbb{E}}$. 
Define $\widetilde{\mathbb{E}}^+$ to be the ring of integers for this valuation.
Put $\widetilde{\mathbb{A}}^{+}=W(\widetilde{\mathbb{E}}^{+})$ and
$\widetilde{\mathbb{B}}^{+}=\widetilde{\mathbb{A}}^{+}[1/p]=\verb+{+\sum_{k\gg -\infty}   p^k[x_k]\ |\ x_k\in \widetilde{\mathbb{E}}^{+}\verb+}+ $
where $[*]$ denotes the Teichm\"uller lift of $*\in \widetilde{\mathbb{E}}^+$.
This ring $\widetilde{\mathbb{B}}^{+}$ is equipped with a surjective homomorphism 
$$\theta:\widetilde{\mathbb{B}}^{+}\twoheadrightarrow\mathbb{C}_{p}:\quad  \sum p^k[x_k]\mapsto   \sum p^k x_{k}^{(0)}.$$
Let $\widetilde{p}$ denote $(p^{(n)})\in \widetilde{\mathbb{E}}^{+}$ such that $p^{(0)}=p$.  Then, $\Ker(\theta)$ is the principal ideal generated by $\omega=[\widetilde{p}]-p$ .
The ring $B_{\dR}^{+}$ is defined to be the $\Ker(\theta)$-adic completion of $\widetilde{\mathbb{B}}^{+}$
$$B_{\dR}^{+}=\underleftarrow{\lim}_{n\geq 0}\widetilde{\mathbb{B}}^{+}/(\Ker(\theta)^n).$$
This is a discrete valuation ring and $t=\log([\epsilon])$ which converges in $B_{\dR}^{+}$ is a generator of the maximal ideal.
Put $B_{\dR}=B_{\dR}^{+}[1/t].$
The ring $B_{\dR}$ becomes a field and is equipped with an action of the Galois group $G_{K}$ and a filtration defined by $\Fil^{i}{B_{\dR}}=t^{i}B_{\dR}^{+}$ ($i\in \mathbb{Z}$).
Then, $(B_{\dR})^{G_{K}}$ is canonically isomorphic to $K$. Thus, for a  $p$-adic representation $V$ of $G_{K}$,  $D_{\dR}(V)=({B_{\dR}}\otimes _{\mathbb{Q}_p}V)^{G_{K}}$ is naturally a $K$-vector space. We say that a $p$-adic representation $V$ of $G_{K}$ is a de Rham representation of $G_{K}$ if we have
$$\dim_{\mathbb{Q}_p}V=\dim_{K}D_{\dR}(V)\quad \text{(we always have $\dim_{\mathbb{Q}_p}V\geq \dim_{K}D_{\dR}(V)$)}.$$

Define $B_{\HT}$ to be the associated graded algebra to the filtration $\Fil^{i}B_{\dR}$.
The quotient $\gr^{i}B_{\HT} = \Fil^{i}B_{\dR}/\Fil^{i+1}B_{\dR}$ $(i \in \mathbb{Z})$ is a one-dimensional
$\mathbb{C}_p$-vector space spanned by the image of $t^i$. Thus, we obtain the presentation
$$B_{\HT} =\bigoplus_{i\in\mathbb{Z}} \mathbb{C}_{p}(i)$$
where $\mathbb{C}_p(i) = \mathbb{C}_p 
 \otimes \mathbb{Z}_p(i)$ is the Tate twist. Then, $(B_{\HT})^{G_K}$ is canonically
isomorphic to $K$. Thus, for a $p$-adic representation $V$ of $G_K$, $D_{\HT}(V )=
(B_{\HT}\otimes _{\mathbb{Q}_{p}}V)^{G_{K}}$ 
is naturally a $K$-vector space. We say that a $p$-adic representation
$V$ of $G_K$ is a Hodge-Tate representation of $G_K$ if we have
$$\dim_{\mathbb{Q}_{p}}V = \dim_{K}D_{\HT}(V )\quad  (\text{we always have $\dim_{\mathbb{Q}_p}V \geq  \dim_{K}D_{\HT}(V )$}).$$

Let $\theta:\widetilde{\mathbb{A}}^{+}\rightarrow \mathscr{O}_{\mathbb{C}_{p}}$ be the natural homomorphism where $\mathscr{O}_{\mathbb{C}_{p}}$ denotes the ring of integers of $\mathbb{C}_{p}$. 
Define the ring $A_{\cris}$ to be the $p$-adic completion of the PD-envelope of $\Ker(\theta)$ compatible with the canonical PD-envelope over the ideal generated by $p$.
Put $B_{\cris}^{+}=A_{\cris}[1/p]$ and $B_{\cris}=B_{\cris}^{+}[1/t]$.
These rings are $K_{0}$-algebras endowed with actions of $G_{K}$ and Frobenius $\varphi$. 
Furthermore, since these rings are canonically included in $B_{\dR}$, they are endowed with the filtration induced by that of $B_{\dR}$. 
Then, $(B_{\cris})^{G_{K}}$ is canonically isomorphic to $K_{0}$. Thus, for a $p$-adic representation $V$ of $G_{K}$, $D_{\cris}(V)=(B_{\cris}\otimes_{\mathbb{Q}_{p}}V)^{G_{K}}$ is naturally a $K_{0}$-vector space. We say that a $p$-adic representation $V$ of $G_{K}$ is a crystalline representation of $G_{K}$ if we have 
$$\dim_{\mathbb{Q}_p}V=\dim_{K_{0}}D_{\cris}(V)\quad \text{(we always have $\dim_{\mathbb{Q}_p}V\geq \dim_{K_{0}}D_{\cris}(V)$)}.$$

Fix a prime element $\pi$ of $\mathscr{O}_{K}$ (the ring of integers of $K$) and an element $s=(s^{(n)})\in \widetilde{\mathbb{E}}^{+}$ such that $s^{(0)}=\pi$.
Then, the series $\log([s]\pi^{-1})$ converges  to an element $u_{\pi}$ in $B_{\dR}^{+}$ and the subring $B_{\cris}[u_{\pi}]$ of $B_{\dR}$ depends only on the choice of $\pi$.
We denote this ring by $B_{\st}$. Since this ring is included in $B_{\dR}$, it is endowed with the action of $G_{K}$ and the filtration induced by those on $B_{\dR}$. 
The element $u_{\pi}$ is transcendental over $B_{\cris}$ and we extend the Frobenius $\varphi$ on  $B_{\cris}$ to  $B_{\st}$ by putting $\varphi(u_{\pi})=pu_{\pi}$. 
Furthermore, define the $B_{\cris}$-derivation $B_{\st}\rightarrow B_{\st}$ by $N(u_{\pi})=-1$.
It is easy to verify that we have $N\varphi=p\varphi N$ and that the action of $N$ on $D_{\st}(V)$ is nilpotent.
As in the case of $B_{\cris}$, we have $(B_{\st})^{G_{K}}=K_{0}$. Thus, for a $p$-adic representation $V$ of $G_{K}$,  $D_{\st}(V)=(B_{\st}\otimes_{\mathbb{Q}_{p}}V)^{G_{K}}$ is naturally a $K_{0}$-vector space. We say that a $p$-adic representation $V$ of $G_{K}$ is a semi-stable representation of $G_{K}$ if we have 
$$\dim_{\mathbb{Q}_p}V=\dim_{K_{0}}D_{\st}(V)\quad \text{(we always have $\dim_{\mathbb{Q}_p}V\geq \dim_{K_{0}}D_{\st}(V)$)}.$$
Furthermore, we say that $V$ is a potentially semi-stable representation of $G_{K}$ if there exists a finite field extension $L/K$ in $\overline{K}$ such that $V$ is a semi-stable representation of $G_{L}$. Due to the result of  Berger [Be1], it is known that $V$ is a potentially semi-stable representation of $G_{K}$ if and only if $V$ is a de Rham representation of $G_{K}$.
\subsection{The theory of Sen}
Keep the notation and assumption in Introduction. In the article [S3], Sen shows that, for a $p$-adic representation $V$ of $G_{K}$, the $\hat{K}_{\infty}(=(\mathbb{C}_{p})^{H_{K}})$-vector space $(\mathbb{C}_{p}\otimes_{\mathbb{Q}_{p}}V)^{H_{K}}$ has dimension $d  =\dim_{\mathbb{Q}_{p}}V$ and the union of the finite dimensional $K_{\infty}$-subspaces of $(\mathbb{C}_{p}\otimes _{\mathbb{Q}_{p}}V)^{H_{K}}$ stable under $\Gamma_{K}$ is a $K_{\infty}$-vector space of dimension $d$ stable under $\Gamma_{K}$ (called $D_{\Sen}(V)$).
We have $\mathbb{C}_{p}\otimes_{K_{\infty}}D_{\Sen}(V) =\mathbb{C}_{p}\otimes_{\mathbb{Q}_{p}}V$ and the natural map $\hat{K}_{\infty}\otimes_{K_{\infty}}D_{\Sen}(V)\rightarrow (\mathbb{C}_{p}\otimes_{\mathbb{Q}_{p}}V)^{H_{K}}$ is an isomorphism. Furthermore, if $\gamma\in \Gamma_K$ is close enough to $1$, then the series of operators on $D_{\Sen}(V)$
$$\frac{\log(\gamma)}{\log(\chi(\gamma))}=-\frac{1}{\log(\chi(\gamma))}\sum_{k\geq 1}\frac{(1-\gamma)^{k}}{k}$$
converges to a $K_{\infty}$-linear derivation $\nabla^{(0)}:D_{\Sen}(V)\rightarrow D_{\Sen}(V)$ and does not depend on the choice of $\gamma$. By the following proposition, we can see that the set of eigenvalues of $\nabla^{(0)}$ on $D_{\Sen}(V)$ is exactly the same as the set of Hodge-Tate weights of $V$ if $V$ is a Hodge-Tate representation of $G_{K}$. 
\begin{prop}
If $V$ is a Hodge-Tate representation of $G_{K}$, there exists a $\Gamma_{K}$-equivariant isomorphism of $K_{\infty}$-vector spaces
$$D_{\Sen}(V)\simeq \bigoplus _{j=1}^{d=\dim_{\mathbb{Q}_{p}}V}K_{\infty}(n_{j})\quad (n_{j}\in\mathbb{Z}).$$
\end{prop}
\begin{proof}
Since $V$ is a Hodge-Tate representation of $G_{K}$, there exists a basis $\verb+{+g_j\verb+}+_{j=1}^{d}$ 
of $\mathbb{C}_{p}\otimes_{\mathbb{Q}_{p}} V$ over $\mathbb{C}_{p}$ such that it gives the Hodge-Tate decomposition
$$\mathbb{C}_{p}\otimes_{\mathbb{Q}_{p}}V\simeq \bigoplus_{j=1}^{d}\mathbb{C}_{p}(n_{j}):g_{j}\mapsto t^{n_{j}} \ (n_j\in\mathbb{Z}).$$
From this presentation, it follows that $\verb+{+g_j\verb+}+^{d}_{j=1}$ forms a basis of a $K_{\infty}$-vector
space $X$ which is contained in $(\mathbb{C}_{p}\otimes_{\mathbb{Q}_{p}}V)^{H_{K}}$ and stable under the action of $\Gamma_{K}$.
Then, since we have $X\hookrightarrow D_{\Sen}(V)$ by definition and both sides have the same
dimension $d$ over $K_{\infty}$, we get the equality $X = D_{\Sen}(V)$. Thus, we obtain the
$\Gamma_{K}$-equivariant isomorphism of $K_{\infty}$-vector spaces $D_{\Sen}(V)\simeq \bigoplus _{j=1}^{d}K_{\infty}(n_{j}):g_{j}\mapsto t^{n_{j}}$.
\end{proof}
\section{Generalization of $D_{\Sen}(V)$}
Let us recall notations. Fix a prime $\pi$ of $\mathscr{O}_{K}$ (the ring of integers of $K$) and for each $1\leq m$, fix a $p^{m}$-th root $\pi^{1/p^{m}}$ of $\pi$ in $\overline{K}$ satisfying $(\pi^{1/p^{m+1}})^{p}=\pi^{1/p^{m}}$. Put 
$K^{\BK}=\cup_{0\leq m}K(\pi^{1/p^{m}}) \ \text{and} \  K_{\infty}^{\BK}=\cup_{0\leq m}K_{\infty}(\pi^{1/p^{m}}).$
Let $H$ denote the Galois group $\Gal(\overline{K}/K_{\infty}^{\BK})$ and define $\Gamma_{\BK}$ to be   $\Gal(K_{\infty}^{\BK}/K_{\infty})$. Then, we have an isomorphism of profinite groups 
$$G_{K}/H\simeq \Gamma_{K}\ltimes \Gamma_{\BK}.$$  
For $\beta\in\Gamma_{\BK}$, we have $\beta(\zeta_{p^{m}})=\zeta_{p^{m}}$ and define the homomorphism
$c:\Gamma_{\BK} \rightarrow \mathbb{Z}_{p}$ such that we have $\beta(\pi^{1/p^{m}})=\pi^{1/p^{m}}\zeta_{p^{m}}^{c(\beta)}$. Then, the homomorphism
$c$ defines an isomorphism $\Gamma_{\BK}\simeq \mathbb{Z}_{p}$ of profinite groups. 
\subsection{Construction of $D_{\pi\text{-}\Sen}(V)$}
Let $V$ be a semi-stable representation of $G_{K}$. For simplicity, assume that the number of the nilpotent block of the monodromy operator $N$ on $D_{\st}(V)$ is $1$. In the general case, one can easily construct $D_{\pi\text{-}\Sen}(V)$ in the same way. Then, there exists a basis $\verb+{+g_{j}\verb+}+_{j}$ of $D_{\st}(V)$ over $K_{0}$ such that we have
$$g_{1}\overset{N}{\longrightarrow} g_{2}\overset{N}{\longrightarrow} \cdots \overset{N}{\longrightarrow} g_{d}\overset{N}{\longrightarrow}0.$$
By twisting $\verb+{+g_{j}\verb+}+_{j}$ by some powers of $t$ in $B_{\st}\otimes_{\mathbb{Q}_{p}}V$, we obtain a basis $\verb+{+f_{j}\verb+}+_{j}$ of $B_{\cris}^{+}[u_{\pi}/t]\otimes_{\mathbb{Q}_{p}}V$ over $B_{\cris}^{+}[u_{\pi}/t]$.
Then, we can write
$$(*)\begin{cases} \circ \ f_{1}=t^{\tiny{m_{1}(=0)}}(F_{1}+(-1)^{1}\frac{u_{\pi}^{1}}{1!t^{1}}F_{2}+(-1)^{2}\frac{u_{\pi}^{2}}{2!t^{2}}F_{3}+\cdots+(-1)^{d-1}\frac{u_{\pi}^{d-1}}{(d-1)!t^{d-1}}F_{d}) \\
\circ \ f_{2}=\mspace{135mu}t^{m_{2}}(F_{2}+(-1)^{1}\frac{u_{\pi}^{1}}{1!t^{1}}F_{3}+\cdots+(-1)^{d-2}\frac{u_{\pi}^{d-2}}{(d-2)!t^{d-2}}F_{d}) \\
\mspace{350mu}\cdots\\
\circ \ f_{d}=\mspace{458mu}t^{m_{d}}F_{d}
\end{cases}$$ 
where $\verb+{+F_{j}\verb+}+_{j}$ denotes a set of elements of $B_{\cris}^{+}\otimes_{\mathbb{Q}_{p}}V$ and we take $m_{j}\in\mathbb{Z}$ such that $\verb+{+f_{j}\verb+}+_{j}$ forms a basis of $B_{\cris}^{+}[u_{\pi}/t]\otimes_{\mathbb{Q}_{p}}V$ over $B_{\cris}^{+}[u_{\pi}/t]$.
\begin{defn}
With notations as above, let $\verb+{+h_{j}:=\overline{t^{m_{j}}F_{j}}\verb+}+_{j}$ denote the image of $\verb+{+t^{m_{j}}F_{j}\verb+}+_{j}$ by the homomorphism $B_{\cris}^{+}\otimes_{\mathbb{Q}_{p}}V\twoheadrightarrow \mathbb{C}_{p}\otimes_{\mathbb{Q}_{p}}V$. Then, define $D_{\pi\text{-}\Sen}(V)$ to be the $K_{\infty}^{\BK}$-vector space generated by $\verb+{+h_{j}\verb+}+_{j}$ contained in $(\mathbb{C}_{p}\otimes_{\mathbb{Q}_{p}}V)^{H}$.
\end{defn} 
\begin{lem}
The elements $\verb+{+h_{j}\verb+}+_{j}$ are linearly independent over $\mathbb{C}_{p}$ in $\mathbb{C}_{p}\otimes_{\mathbb{Q}_{p}}V$. In particular, $\verb+{+h_{j}\verb+}+_{j}$ forms a basis of $D_{\pi\text{-}\Sen}(V)$ over $K_{\infty}^{\BK}$ and its dimension over $K_{\infty}^{\BK}$ is equal to $\dim_{\mathbb{Q}_{p}}V$.
\end{lem}
\begin{proof}
We can show inductively that $\verb+{+h_{d}\verb+}+$, $\verb+{+h_{d-1},h_{d}\verb+}+$, ..., $\verb+{+h_{1},\ldots,h_{d}\verb+}+$ are linearly independent over $\mathbb{C}_{p}$ in $\mathbb{C}_{p}\otimes_{\mathbb{Q}_{p}}V$.
\end{proof}
By this lemma, we can easily verify that the following proposition holds.
\begin{prop}(c.f. Subsection $2.2$)
We have $\mathbb{C}_{p}\otimes_{K_{\infty}^{\BK}}D_{\pi\text{-}\Sen}(V) =\mathbb{C}_{p}\otimes_{\mathbb{Q}_{p}}V$ and the natural map $(\mathbb{C}_{p})^{H}\otimes_{K_{\infty}^{\BK}}D_{\pi\text{-}\Sen}(V)\rightarrow (\mathbb{C}_{p}\otimes_{\mathbb{Q}_{p}}V)^{H}$ is an isomorphism.
\end{prop}  
It follows easily that the $K_{\infty}^{\BK}$-vector space $D_{\pi\text{-}\Sen}(V)$ is equipped with the $K_{\infty}^{\BK}$-linear derivation $\nabla^{(0)}=\frac{\log(\gamma)}{\log(\chi(\gamma))}$ if $\gamma\in\Gamma_{K}$ is close enough to $1$. By the following proposition, we can see that the action of $\nabla^{(0)}$ on $D_{\pi\text{-}\Sen}(V)$ tells us about Hodge-Tate weights as in the case of $D_{\Sen}(V)$.
\begin{prop}(c.f. Proposition $2.1$) For a semi-stable representation $V$ of $G_{K}$, there exists a $\Gamma_{K}$-equivariant isomorphism of $K_{\infty}^{\BK}$-vector spaces
$$D_{\pi\text{-}\Sen}(V)\simeq \bigoplus_{j=1}^{d} K_{\infty}^{\BK}(n_{j}): h_{j}\mapsto t^{n_{j}} \quad (n_{j}\in\mathbb{Z}).$$
Furthermore, the set of integers $\verb+{+n_{j}\verb+}+_{j}$ is exactly the same as the set of Hodge-Tate weights of $V$.
\end{prop} 
\begin{proof}Note that we have $\verb+{+\gamma(f_{j})=\chi(\gamma)^{n_{j}}f_{j}\verb+}+_{j}$ by definition. Then, we can show inductively that we have $\verb+{+\gamma(h_{d})=\chi(\gamma)^{n_{d}}h_{d}\verb+}+$, $\verb+{+\gamma(h_{d-1})=\chi(\gamma)^{n_{d-1}}h_{d-1}\verb+}+,$ $\ldots$, $\verb+{+\gamma(h_{1})=\chi(\gamma)^{n_{1}}h_{1}\verb+}+$. The rest is easily verified by Proposition $3.3$. 
\end{proof}  
On the other hand, if $\beta\in\Gamma_{\BK}$ is close enough to $1$, the series of operators on $D_{\pi\text{-}\Sen}(V)$
$$\nabla^{(\pi)}=\frac{\log(\beta)}{c(\beta)}=-\frac{1}{c(\beta)}\sum_{k\geq 1}\frac{(1-\beta)^{k}}{k}$$
converges to a $K_{\infty}^{\BK}$-linear derivation on $D_{\pi\text{-}\Sen}(V)$ does not depend on the choice of $\beta\in \Gamma_{\BK}$. This easily follows from the calculations $\nabla^{(\pi)}(f_{j})=0$ and $\nabla^{(\pi)}(\frac{u_{\pi}}{t})=1$.
\begin{rem}
By using the calculations $\nabla^{(\pi)}(f_{j})=0$ and $\nabla^{(\pi)}(\frac{u_{\pi}}{t})=1$, we obtain $\nabla^{(\pi)}(F_{j})=F_{j+1}$ ($j<d$) and $\nabla^{(\pi)}(F_{d})=0$. Thus, we can rewrite $(*)$ as 
$$(*)' \quad f_{j}=t^{m_{j}}(\sum_{k=0}^{d-j}(-1)^{k}\frac{u_{\pi}^{k}}{k!t^{k}}(\nabla^{(\pi)})^{k}(F_{j})) \quad  (1\leq j\leq d).$$
Compare this formula to the main construction  $\verb+{+f_{j}^{(*)}\verb+}+_{j}$ in [M1]. In fact, the idea of the construction of $D_{\pi\text{-}\Sen}(V)$ is based on the similarity between Corollary 2.1.14 of [Ki] and Main Theorems of [M1] and [M2].
\end{rem}
\subsection{Some properties of differential operators} 
We shall describe the actions of derivations $\nabla^{(0)}$ and $\nabla^{(\pi)}$ on $D_{\pi\text{-}\Sen}(V)$. First, by a standard argument, we can show that, if $x \in D_{\pi\text{-}\Sen}(V)$, we have
$$
\nabla^{(0)}(x)=\lim_{\gamma\rightarrow 1}\frac{\gamma(x)-x}{\chi(\gamma)-1}\quad \text{and} \quad \nabla^{(\pi)}(x)=\lim_{\beta\rightarrow 1}\frac{\beta(x)-x}{c(\beta)}.
$$
By using these presentations, we compute the bracket  $[\ , \ ]$ of derivations $\nabla^{(0)}$ and $\nabla^{(\pi)}$ on $D_{\pi\text{-}\Sen}(V)$.
\begin{prop} On the differential module $D_{\pi\text{-}\Sen}(V)$, we
have $[\nabla^{(0)},\nabla^{(\pi)}]= \nabla^{(\pi)}$.
\end{prop}
\begin{proof} First, note that we have the relation 
$\gamma\beta =\beta^{\chi(\gamma)}\gamma$. Then, since we have
$$\lim _{h\rightarrow 0}
\frac{a^{h+1} - a}{(h + 1) - 1}=a \log(a),$$ 
we obtain
{\small \begin{align*}
[\nabla^{(0)},\nabla^{(\pi)}](*) =& \lim_{\gamma\rightarrow 1}\frac{\gamma-1}{\chi(\gamma)-1}\lim_{\beta\rightarrow 1}\frac{\beta-1}{c(\beta)}(*)-\lim_{\beta\rightarrow 1}\frac{\beta-1}{c(\beta)}\lim_{\gamma\rightarrow 1}\frac{\gamma-1}{\chi(\gamma)-1}(*)\\
=&\lim_{\beta\rightarrow 1}\lim_{\gamma\rightarrow 1}\frac{\gamma\beta-\gamma-\beta+1}{(\chi(\gamma)-1)c(\beta)}(*)-\lim_{\beta\rightarrow 1}\lim_{\gamma\rightarrow 1}\frac{\beta\gamma-\gamma-\beta+1}{(\chi(\gamma)-1)c(\beta)}(*)\\
=&\lim_{\beta\rightarrow 1}\lim_{\gamma\rightarrow 1}\frac{\beta^{\chi(\gamma)}\gamma-\beta\gamma}{(\chi(\gamma)-1)c(\beta)}(*)\\
=&\lim_{\beta\rightarrow 1}\frac{\beta\log(\beta)}{c(\beta)}(*)\\
=&\nabla^{(\pi)}(*).
\end{align*}}
\end{proof}
\begin{prop} The action of the $K_{\infty}^{\BK}$-linear derivation $\nabla^{(\pi)}$ on
$D_{\pi\text{-}\Sen}(V )$ is nilpotent.
\end{prop}
\begin{proof} From the equality $\nabla^{(0)}\nabla^{(\pi)}-\nabla^{(\pi)}\nabla^{(0)} = \nabla^{(\pi)}$, we get $\nabla^{(0)}(\nabla^{(\pi)})^{r}-(\nabla^{(\pi)})^{r}\nabla^{(0)}$  $= r(\nabla^{(\pi)})^{r}$ and $\tr(r(\nabla^{(\pi)})^{r}) = 0$ for all $r\in\mathbb{N}$. Since the characteristic of $K_{\infty}^{\BK}$ is $0$,
we obtain $\tr((\nabla^{(\pi)})^{r}) = 0$ for all $r\in\mathbb{N}$. As is well known in linear algebra, this
shows that the action of the $K_{ \infty}^{\BK}$-linear derivation $\nabla^{(\pi)}$ on
$D_{\pi\text{-}\Sen}(V )$ is nilpotent.
\end{proof}
\begin{prop}
For an element $x\in D_{\pi\text{-}\Sen}(V)$ such that $\nabla^{(0)}(x)=nx$ ($n\in\mathbb{Z}$), we have $\nabla^{(0)}(\nabla^{(\pi)}(x))=(n+1)\nabla^{(\pi)}(x)$, that is, the action of $\nabla^{(\pi)}$ increases the Hodge-Tate weight by $1$. 
\end{prop}
\begin{proof} This follows easily from the relation $[\nabla^{(0)},\nabla^{(\pi)}]=\nabla^{(\pi)}$.
\end{proof}
There are many choices of a $K_{\infty}^{\BK}$-subspace of dimension $\dim_{\mathbb{Q}_{p}}V$ in $(\mathbb{C}_{p}\otimes_{\mathbb{Q}_{p}}V)^{H}$ equipped with derivations $\nabla^{(0)}$ and $\nabla^{(\pi)}$. The aim of this article is, however, to construct a differential module in $\mathbb{C}_{p}\otimes_{\mathbb{Q}_{p}}V$ which is closely related to the module $D_{\st}(V)$. Thus, the following proposition says that the choice $D_{\pi\text{-}\Sen}(V)$ may be a reasonable one. 
\begin{prop}For a crystalline representation $V$ of $G_{K}$, the action of $\nabla^{(\pi)}$ on $D_{\pi\text{-}\Sen}(V)$ is trivial.
\end{prop}
\begin{proof}In the case when $V$ is a crystalline representation of $G_{K}$, we can take $\verb+{+\overline{f_{j}}\verb+}+_{j}$ as a basis of $D_{\pi\text{-}\Sen}(V)$ over $K_{\infty}^{\BK}$.
We can see that the action of $\Gamma_{\BK}$ on this basis is trivial and thus the action of $\nabla^{(\pi)}$ on $D_{\pi\text{-}\Sen}(V)$ is trivial.
\end{proof}
Conversely, there is a semi-stable representation $V$ of $G_{K}$ such that the action of $\nabla^{(\pi)}$ on $D_{\pi\text{-}\Sen}(V)$ is non-trivial. The next example is the prototype of such a semi-stable representation.
\begin{ex}
Let $V$ be a $p$-adic representation of $G_{K}$ attached to the Tate curve $\overline{K}^{*}/\langle\pi\rangle$. We can take a basis $\verb+{+e,f\verb+}+$ of $V$ over $\mathbb{Q}_{p}$ such that the action of $g\in G_{K}$  is given by
\begin{equation*}
\begin{pmatrix}
\chi(g) &c(g)\\
0       &1
\end{pmatrix}.
\end{equation*}
It is easy to see that $\verb+{+h_{1}=1\otimes f,h_{2}=1\otimes e\verb+}+ \ (\subset\mathbb{C}_{p}\otimes_{\mathbb{Q}_{p}}V)$ forms a basis of $D_{\pi\text{-}\Sen}(V)$ over $K_{\infty}^{\BK}$. As indicated by Proposition $3.4$, we have $$\nabla^{(0)}(h_{1})=0\quad \text{and}  \quad \nabla^{(0)}(h_{2})=h_{2},$$ that is, the Hodge-Tate weights of $V$ are $\verb+{+0, 1\verb+}+$. Furthermore, the action of $\nabla^{(\pi)}$ on this basis is given by 
$$h_{1}\overset{\nabla^{(\pi)}}{\longrightarrow}h_{2}\overset{\nabla^{(\pi)}}{\longrightarrow}0.$$
This means that the action of $\nabla^{(\pi)}$ on $D_{\pi\text{-}\Sen}(V)$ is nilpotent (Proposition 3.7) and that the action of $\nabla^{(\pi)}$ increases the Hodge-Tate weights by $1$ (Proposition 3.8). Thus, we can know more than Hodge-Tate weights by using the $K_{\infty}^{\BK}$-vector space $D_{\pi\text{-}\Sen}(V)$ equipped with $\nabla^{(\pi)}$. 
\end{ex} 
\section{Geometric aspect of $D_{\pi\text{-}\Sen}(V)$}
Let $X$ be a proper smooth scheme over $K$. Then, it is known that the $p$-adic \'etale cohomology $V^{m}=H^{m}_{\text{\'et}}(X\otimes_{K}\overline{K},\mathbb{Q}_{p})$ is a de Rham representation of $G_{K}$.
Furthermore, due to the result of Berger, we can see that $V^{m}$ is a potentially semi-stable representation of $G_{K}$.
Let $L/K$ be a finite field extension of $K$ in $\overline{K}$ such that $V^{m}$ is a semi-stable representation of $G_{L}$ and let $V_{L}^{m}$ denote the restriction of $V^{m}$ to $G_{L}$. 

In this section, we shall study the geometric aspect of $D_{\pi\text{-}\Sen}(V_{L}^{m})$ and see that the action of $\nabla^{(\pi)}$ describes an analogy of the infinitesimal variations of Hodge structures and satisfies  formulae similar to the Griffiths transversality and the local monodromy theorem.  
First, by Proposition 3.4, we obtain the $\Gamma_K$-equvariant isomorphism of $L_{\infty}^{\BK}$-vector spaces 
$D_{\pi\text{-}\Sen}(V_{L}^{m})\simeq \bigoplus_{j=1}^{\dim_{\mathbb{Q}_{p}}V}L_{\infty}^{\BK}(n_{j}).$
With this presentation, define the subspace $D^{s,t}_{\pi\text{-}\Sen}(V_{L}^{m})$ of $D_{\pi\text{-}\Sen}(V_{L}^{m})$ to be $D^{m-t,t}_{\pi\text{-}\Sen}(V_{L}^{m})=\verb+{+x\in D_{\pi\text{-}\Sen}(V_{L}^{m})\verb+|+\nabla^{(0)}(x)=tx\verb+}+$ ($t\in\mathbb{Z}$). It follows easily that we obtain the decomposition
$$D_{\pi\text{-}\Sen}(V_{L}^{m})=\bigoplus_{s+t=m}D^{s,t}_{\pi\text{-}\Sen}(V_{L}^{m}).$$
The next proposition claims that the action of $\nabla^{(\pi)}$ on $D_{\pi\text{-}\Sen}(V_{L}^{m})$ satisfies a formula similar to Griffiths transversality.
\begin{prop}(Transversality) With notations as above, we have
$$\nabla^{(\pi)}(D^{s,t}_{\pi\text{-}\Sen}(V_{L}^{m}))\subset D^{s-1,t+1}_{\pi\text{-}\Sen}(V_{L}^{m}).$$
\end{prop}
\begin{proof}
This follows easily from Proposition 3.8.
\end{proof}
By the same argument, we can see that an analogy of the local monodromy theorem holds for the $L_{\infty}^{\BK}$-vector space $D_{\pi\text{-}\Sen}(V_{L}^{m})$ equipped with $\nabla^{(\pi)}$.
\begin{prop}(Local monodromy theorem) With notations as above, the $L_{\infty}^{\BK}$-linear operator $\nabla^{(\pi)}$ satisfies
$$(\nabla^{(\pi)})^{m+1}\verb+|+ D_{\pi\text{-}\Sen}(V_{L}^{m})=0.$$
Furthermore, if we put $h^{s,t}=\dim_{L_{\infty}^{\BK}}D^{s,t}_{\pi\text{-}\Sen}(V_{L}^{m})$ and define  $h_{m}=\sup\verb+{+b-a\verb+|+\forall i\in [a,b], h^{i,m-i}\not=0\verb+}+$, we have
$$(\nabla^{(\pi)})^{h_{m}+1}\verb+|+ D_{\pi\text{-}\Sen}(V_{L}^{m})=0.$$
\end{prop}

\end{document}